\documentclass[letterpaper,11pt]{amsart}

\usepackage[margin=1.2in]{geometry}
\usepackage{amsmath,amsthm,amssymb}
\usepackage{xspace,xcolor}
\usepackage[
  breaklinks,colorlinks,
  citecolor=teal,linkcolor=teal,urlcolor=teal
]{hyperref}
\usepackage[alphabetic]{amsrefs}
\usepackage[all]{xy}
\usepackage{booktabs}

\theoremstyle{plain}
\newtheorem{thm}{Theorem}[section]

\newtheorem{lem}[thm]{Lemma}

\theoremstyle{definition}

\theoremstyle{remark}
\newtheorem{rmk}[thm]{Remark}
\newtheorem{eg}[thm]{Example}

\numberwithin{equation}{section}

\def\Z{{\mathbb Z}}
\def\Q{{\mathbb Q}}

\def\C{{\mathbb C}}

\def\A{{\mathbb A}}
\def\P{{\mathbb P}}
\def\F{{\mathbb F}}

\def\cE{\mathcal{E}}
\def\cF{\mathcal{F}}

\def\cI{\mathcal{I}}

\def\cN{\mathcal{N}}
\def\cO{\mathcal{O}}
\def\cP{\mathcal{P}}

\def\cS{\mathcal{S}}

\def\cX{\mathcal{X}}
\def\cY{\mathcal{Y}}
\def\cZ{\mathcal{Z}}

\def\E{\mathcal{E}}
\def\F{\mathcal{F}}

\def\O{\mathcal{O}}
\def\X{\mathcal{X}}
\def\Y{\mathcal{Y}}
\def\Z{\mathcal{Z}}

\def\fa{\mathfrak{a}}

\def\fm{\mathfrak{m}}

\def\a{\alpha}

\def\g{\gamma}

\def\f{\phi}
\def\ff{\psi}

\def\p{\pi}

\def\s{\sigma}
\def\t{\tau}

\def\.{\cdot}
\def\^{\widehat}
\def\~{\widetilde}
\def\o{\circ}
\def\ov{\overline}

\def\({\left(}
\def\){\right)}

\renewcommand{\and}{ \ \ \text{ and } \ \ }

\newcommand{\fall}{ \ \ \text{ for all } \ \ }

\newcommand{\where}{ \ \ \text{ where } \ \ }

\def\reg{\mathrm{reg}}
\def\sing{\mathrm{sing}}

\def\an{\mathrm{an}}

\DeclareMathOperator{\codim} {codim}

\DeclareMathOperator{\Proj} {Proj}
\DeclareMathOperator{\Spec} {Spec}
\DeclareMathOperator{\Spf} {Spf}

\DeclareMathOperator{\Bl} {Bl}
\DeclareMathOperator{\val} {val}

\DeclareMathOperator{\Ex} {Ex}

\DeclareMathOperator{\ord} {ord}

\DeclareMathOperator{\Ext} {Ext}

\begin{document}

%{\hfill\small\tt\jobname\vspace{3em}}

%section{Title and Abstract}

\title     {Three-dimensional counter-examples to the Nash problem}

\author    {Tommaso de Fernex}
\address   {Department of Mathematics,
            University of Utah,
            155 South 1400 East,
            Salt Lake City, UT 84112, USA} 
\email     {defernex@math.utah.edu}

%\dedicatory{Working notes}



\thanks    {Research partially supported by NSF CAREER Grant DMS-0847059
and by the Simons Foundation.}

\thanks{Compiled on \today. Filename \small\tt\jobname }

%\subjclass [2010]{Primary 14E05; Secondary 14E18, 14M22, 14D06}
%\keywords  {Arcs, divisorial valuations, resolution of singularities}

\begin{abstract}
The Nash problem asks about the existence of a correspondence between
families of arcs through singularities of complex
varieties and certain types of divisorial valuations. 
It has been positively settled in dimension 2
by Fern\'andez de Bobadilla and Pe Pereira, 
and it was shown to have a negative answer in all dimensions $\ge 4$
by Ishii and Koll\'ar.
In this note we discuss examples which show that the
problem has a negative answer in dimension 3 as well.
These examples bring also to light the different nature of the problem
depending on whether it is formulated in
the algebraic setting or in the analytic setting.
\end{abstract}

\maketitle

\section{Introduction}

The space of arcs through the singularities
of a complex variety has finitely many irreducible
components, each of which is naturally associated to
a divisorial valuation of the function field of the variety. 
Every valuation arising in this
way is essential for the singularity, in the sense that its center in 
any resolution of singularities is an irreducible component of the 
inverse image of the singular locus. 
The Nash problem asks whether, conversely, every essential valuation
corresponds to a component of the space of arcs through the singularities.
This summarizes the main content of Nash's
influential paper \cite{Na}, which circulated as a preprint since the late 1960's.

The Nash problem has attracted the attention of the mathematical community for a long time, 
and the surface case has finally been settled by 
Fern\'andez de Bobadilla and Pe Pereira \cite{FP}. 
The problem has however a negative answer in general:
examples of essential divisorial valuations that do not correspond to any irreducible component
of the space of arcs through the singularities were found in all dimensions
$\ge 4$ by Ishii and Koll\'ar \cite{IK}.

The purpose of this note is to extend the class of examples to dimension 3, 
the only dimension not covered by these results. To this end, we study two examples.
We first consider the affine hypersurface in $\A^4$ of equation
\[
(x_1^2 + x_2^2 + x_3^2)x_4 + x_1^3 + x_2^3 + x_3^3 + x_4^5 + x_4^6 = 0.
\]
It turns out that this gives a counter-example to the Nash problem in the category of algebraic 
varieties but not in the category of complex analytic spaces. 
By degenerating the above equation to
\[
(x_2^2 + x_3^2)x_4 + x_1^3 + x_2^3 + x_3^3 + x_4^5 + x_4^6 = 0,
\]
we then obtain an example that works in both categories.

In the first example we provide two arguments 
to show that a certain divisorial valuation 
is not in the Nash correspondence: the first argument
uses a lemma from Ishii and Koll\'ar's paper, and the second one 
relies on the simple observation that such valuation is not essential in the analytic category. 
The same property is then deduced for the second example
by analyzing the deformation of the corresponding family of arcs.
To prove that these valuations are essential in the appropriate categories
we take into account discrepancies and factoriality.
The distinction in the nature of the two examples is a consequence of the
difference of being $\Q$-factorial in the Zariski topology and in the analytic topology.

For clarity of exposition, we will present one example at a time.
Section~\ref{s:IK} is devoted to a discussion of the lemma of Ishii and Koll\'ar.
In the last section we briefly compare our results to the recent paper \cite{Kol2}.

Unless otherwise stated,
we work in the category of algebraic varieties over the field of complex numbers;
the complex analytic setting will also be considered. 
Everything discussed in the algebraic setting holds more generally 
when the ground field is any algebraically
closed field of characteristic zero. 

The author is indebted to Roi Docampo for several precious discussions and suggestions,
and to J\'anos Koll\'ar for valuable remarks
concerning the analytic counterpart of the problem which have been a motivation 
for the second example. He is also grateful to
Charles Favre, Javier Fern\'andez de Bobadilla and Shihoko Ishii for useful comments.
Special thanks go to the referees for their many comments and suggestions which have 
helped improving the exposition of the paper.

\section{The Nash problem}
\label{s:nash-problem}
 
Let $X$ be an algebraic variety, and let $J_\infty(X)$ be the space of formal arcs on $X$. 
This space is defined as the inverse limit of the
jet schemes $J_m(X)$ of $X$, and thus $J_\infty(X)$ is a scheme over $\C$. 
Its Zariski topology coincides with the limit topology \cite[(8.2.10)]{EGAiv}.
For more details on jet schemes and arc spaces, we refer to \cite{Na,DL,EM}.

We mainly regard $J_\infty(X)$ as a topological space, consisting of $\C$-valued arcs:
\[
J_\infty(X) = \{ \f \colon \Spec\C[[t]] \to X \}.
\]
Similarly, each $m$-th jet scheme of $X$, as a set, consists of $\C$-valued $m$-jets:
\[
J_m(X) = \{ \gamma \colon \Spec\C[t]/(t^{m+1}) \to X \}.
\]
There are canonically defined truncation maps 
\[
\p_{X,m} \colon J_\infty(X) \to J_m(X) \and \p_X \colon J_\infty(X) \to X,
\]
where $\p_X = \p_{X,0}$ is the projection that
maps an arc $\f$ to its base point $\f(0) \in X$.

A basic property that we will use several times is
that if $X$ is smooth and $T \subset X$ is an irreducible closed subset, 
then $\p_X^{-1}(T) \subset J_\infty(X)$ is also irreducible.
This property may fail if $X$ is singular. 

Our focus is on the set of arcs through the singular locus $X_\sing$ of $X$, namely, 
the set
\[
\p_X^{-1}(X_\sing) \subset J_\infty(X)
\]
Relevant information about this set can be derived 
by looking at resolutions of singularities.
The following is an overview of some general results on the 
structure of this set, proven in \cite[Theorem~2.15]{IK} and \cite{Na}. 

By definition, a resolution of singularities of $X$ consists of
a proper birational morphism $f \colon Y \to X$ from a smooth variety $Y$.
The exceptional locus $\Ex(f)$ of $f$ is the complement of the largest open set
on which $f$ induces an isomorphism.

Let $f \colon Y \to X$ be a resolution of singularities,
and let $E_1, \dots, E_k$ be the irreducible components of $f^{-1}(X_\sing)$.
For clarity of exposition, assume that $f$
is an isomorphism over the regular locus $X_\reg$ of $X$, 
so that $\Ex(f) = f^{-1}(X_\sing)$.
Then the induced map $f_\infty \colon J_\infty(Y) \to J_\infty(X)$ is
surjective and gives a continuous bijection 
\[
\big(J_\infty(Y) \smallsetminus J_\infty(f^{-1}(X_\sing))\big) \xrightarrow{1-1}
\big(J_\infty(X) \smallsetminus J_\infty(X_\sing)\big).
\]
%The set $J_\infty(X_\sing)$ has infinite codimension in $J_\infty(X)$ 
%(and so does $J_\infty(f^{-1}(X_\sing))$ in $J_\infty(Y)$). 
%From the point of view of motivic integration, this set have
%measure zero. It was in fact excluded altogether in \cite{Na}.
Unless $X$ is smooth and $f$ is an isomorphism, the inverse of this map
is not continuous. The Nash problem is an effort to
gain some understanding of how the topology changes. 

Set-theoretically, we have
\[
f_\infty \big(\p_Y^{-1}(f^{-1}(X_\sing))\big) = \p_X^{-1}(X_\sing).
\]
The set $\p_Y^{-1}(f^{-1}(X_\sing))$ is the
union of the irreducible sets $\p_Y^{-1}(E_j)$, and hence $\p_X^{-1}(X_\sing)$
has finitely many irreducible components $C_1,\dots,C_s$, each 
equal to the closure of some $f_\infty \big(\p_Y^{-1}(E_j)\big)$.
Following the terminology of \cite{Ishii}, we shall 
refer to them as the {\it Nash components}. 
 
The function field $K_i$ of a Nash component $C_i$ is an extension of the function field 
of $X$, and
thus the generic point $\a_i$ of $C_i$, viewed as an arc $\a_i \colon \Spec K_i[[t]] \to X$,
defines a valuation $\val_{\a_i}$ on $X$ (also denoted by $\val_{C_i}$).
If for some $j$ the component $E_j$ has codimension 1 in $Y$
and $f_\infty\big(\p_Y^{-1}(E_j)\big)$ is dense in $C_i$, then $\val_{\a_i}$
is just the divisorial valuation $\val_{E_j}$ given by the order
of vanishing at the generic point of $E_j$. In particular, since
we can pick $f$ such that $f^{-1}(X_\sing)$ has pure
codimension 1, it follows that all valuations associated
to Nash components are divisorial. 
Note that each $C_i$ is dominated by exactly one of the $\p_Y^{-1}(E_j)$. 

In general, if $f \colon Y \to X$ is an arbitrary resolution of singularities,
then every valuation corresponding to a Nash component $C_i$ of $\p_X^{-1}(X_\sing)$
must have center in $Y$ equal to an irreducible component of $f^{-1}(X_\sing)$. 
If we call \emph{essential divisorial valuation}
any divisorial valuation on $X$ whose center on every resolution $f \colon Y \to X$
is an irreducible component of $f^{-1}(X_\sing)$, then we obtain a map
\[
\{\,\text{Nash components of $\p_X^{-1}(X_\sing)$}\,\} 
\longrightarrow
\{\,\text{essential divisorial valuations on $X$}\,\}.
\]
This is the \emph{Nash map} (cf.~\cite{Na}).
It is clear by construction that the Nash map is injective, 
and the question is whether it is surjective.

\section{The analytic setting}
\label{s:analytic}

Resolutions of singularities play an important role in the Nash problem: 
first of all, to prove that there are only finitely many families of 
arcs through the singular locus, and furthermore to define the notion of 
essential divisorial valuation. 
At the time of the writing of \cite{Na}, resolution of singularities 
was not yet known to exist for complex analytic varieties, and 
the treatment was restricted to algebraic varieties. 
The existence of resolutions in the analytic setting was however 
established a few years later, 
and is thus natural to formulate the Nash problem in this setting as well.

The arc space $J_\infty(X)$ of a complex analytic variety $X$ can be defined locally: 
if $X$ is defined by the vanishing of finitely many 
holomorphic functions $h_j(x_1,\dots,x_n)$ in an open domain $U \subset \C^n$,
then $J_\infty(X)$ is the set of $n$-ples of formal power series $x_i(t) \in \C[[t]]$
such that $(x_1(0),\dots,x_n(0)) \in U$ and $h_j(x_1(t),\dots,x_n(t)) \equiv 0$ for all $j$. 
The spaces of jets $J_m(X)$ are defined analogously, and are analytic spaces. 
Just like in the algebraic case, the arc space of an analytic variety 
is the inverse limit of the jet spaces, 
and as such inherits the limit (analytic) topology.

It follows from the local description of arc spaces that if 
$X$ is a complex analytic variety and $X^\an$
is the associated analytic space, then 
the arc spaces $J_\infty(X)$ and $J_\infty(X^\an)$ are in natural bijection, 
and so are their jet spaces. The projections are also compatible, and hence if we 
extend to the analytic setting the notation introduced in the algebraic setting, 
then, under these bijections, we have $\p_{X^\an,m} = \p_{X,m}$ and $\p_{X^\an} = \p_{X}$.
Note also that $(X^\an)_\sing = X_\sing$.

By \cite{Gr}, the image of $\p_{X^\an}^{-1}((X^\an)_\sing)$
in every $J_m(X^\an)$ is a constructible subset, 
and thus its closure has a decomposition into
a union of finitely many irreducible analytic subvarieties
of $J_m(X^\an)$. 
Such decomposition stabilizes for $m \gg 1$, 
and determines, in the limit, a decomposition of $\p_{X^\an}^{-1}((X^\an)_\sing)$
into the union of finitely many closed sets (see \cite{Na}).
We call these sets
the {\it Nash components} of $\p_{X^\an}^{-1}((X^\an)_\sing)$.
This decomposition agrees with the decomposition
of $\p_X^{-1}(X_\sing)$ into Nash components
described in the previous section.

Still, the Nash problem depends on the choice of the category. 
This has to do with the notion of essential divisorial valuation.
The issue is that there may be
resolutions of singularities given by analytic spaces that are not schemes.
We will see this occurring in the first of the two examples 
discussed in this paper.

Every divisorial valuation $\val_E$ on $X$ induces a divisorial valuation on $X^\an$, 
which, with slight abuse of notation, we shall still denote by $\val_E$.
(The converse is also true for valuations with zero-dimensional centers;
this may be relevant if $X$ has isolated singularities.)

If we define {\it essential divisorial valuations}
on an analytic variety analogously as we defined essential divisorial valuations
in the algebraic setting, then the set of essential divisorial 
valuations on an algebraic variety $X$ may differ from the set of 
essential divisorial valuations on the associated analytic variety $X^\an$. 
However, since the spaces through the singularities of $X$ and of $X^\an$ 
are in natural bijection and so are their decompositions into Nash components, 
the image of the Nash map is the same regardless of whether we work in the 
algebraic category or in the analytic category. 
The situation is summarized in the following commutative diagram:
\[
\xymatrix@R=8pt@C=15pt{
\{\,\text{Nash components of $\p_X^{-1}(X_\sing)$}\,\} \ar@{^(->}[r] \ar@{=}[dd] 
& \{\,\text{essential divisorial valuations on $X$}\,\} \ar@{_(->}[d] \\
& \{\,\text{divisorial valuations on $X^\an$}\,\} \\
\{\,\text{Nash components of $\p_{X^\an}^{-1}((X^\an)_\sing)$}\,\} \ar@{^(->}[r] 
& \{\,\text{essential divisorial valuations on $X^\an$}\,\} \ar@{^(->}[u] 
}.
\]

The existence of more resolutions in the analytic category may result,
for instance, from the difference between the notion of factoriality
(or, more generally, of $\Q$-factoriality) in the two topologies.

Recall that a normal algebraic variety $X$ is factorial (resp., $\Q$-factorial) 
in the Zariski topology  
if every Weil divisor on $X$ is Cartier (resp., $\Q$-Cartier).
Since every Weil divisor defined on a Zariski open set of $X$ extends to a divisor on $X$, 
this notion is local in the Zariski topology.
We say that $X^\an$ is (locally) 
factorial (resp., $\Q$-factorial) in the analytic topology if for
every Euclidean open set $U \subset X^\an$, 
every Weil divisor on $U$ is Cartier (resp., $\Q$-Cartier).
Differently from the algebraic case, 
this notion is not global, as there may be Weil divisors on $U$
that do not extend to Weil divisors on $X^\an$. 

The following property must be well-known. We give a proof for completeness. 

\begin{lem}
\label{l:small}
Let $f \colon Y \to X$ be a resolution of singularities of 
a normal algebraic variety (resp., of a normal analytic threefold). 
Assume that $\Ex(f)$ has an irreducible component of codimension $\ge 2$ in $Y$. 
Then $X$ is not $\Q$-factorial in the Zariski topology (resp., in the
analytic topology).
\end{lem}

\begin{proof}
Let $C \subset Y$ be an irreducible curve that is contracted by $f$
but is not contained in any codimension 1 component of $\Ex(f)$.

If $X$ is an algebraic variety, then
we take an affine open set $V \subset Y$ intersecting $C$, 
and let $H' \subset V$ be a general hyperplane section intersecting $C$. 
The Zariski closure of $H'$ in $Y$ produces an effective divisor $H$ that is not 
exceptional for $f$ and satisfies $H \. C > 0$. 
If $f_*H$ were $\Q$-Cartier, then we would have 
$f^*f_*H = H + G$ where $G$ is an effective $f$-exceptional 
$\Q$-divisor, and thus $0 = f^*f_*H \. C = H\.C + G\.C> 0$, a contradiction.
Therefore $f_*H$ is not $\Q$-Cartier, and 
hence $X$ is not $\Q$-factorial in the Zariski topology.

If $X$ is an analytic threefold, then $C$ is an irreducible component of $\Ex(f)$. 
We consider
a small portion of hypersurface $H' \subset Y$ transverse to $C$, defined
locally using some analytic coordinates centered
at a general point of $C$. If $U \subset X$ is a sufficiently small
Euclidean open neighborhood of the point $h(C) \in X$, then
the restriction of $H'$ defines a divisor $H$ on $f^{-1}(U)$
that is not exceptional over $U$ and satisfies $H \. C > 0$, 
and the conclusion follows as in the algebraic case. 
\end{proof}

\begin{eg}
\label{eg:small}
A typical example of a variety that is factorial in the Zariski topology but
is not even $\Q$-factorial in the analytic topology is given by hypersurfaces in $\P^4$ 
with some (but not too many) ordinary double points.
Locally analytically, any such threefold $X$ is isomorphic to $xy=zx$ near a singular point $P$.
The ideal $(x,z)$ is not principal but its vanishing 
defines a divisor locally near $P$, hence $X^\an$ is not factorial 
in the analytic topology.
If we blow-up this ideal near each singular point, and glue all charts back together, 
we obtain a small resolution $X' \to X^\an$ in the analytic category,
hence $X^\an$ is not even $\Q$-factorial by Lemma~\ref{l:small}.
On the other hand, if the number of points is small with respect to the degree, 
then $X$ is factorial in the Zariski topology, see \cite{Cy,CDG,Che}.
In particular, $X'$ cannot be a scheme. The point is that the divisor defined locally by $(x,z)$
near $P$ does not extend to a global divisor on $X$.
This particular class of examples will appear below in the discussion of the first example.
\end{eg}

\section{Some results on the Nash problem}

The main approach towards the Nash 
problem was first proposed by Lejeune-Jalabert in \cite{LJ},
and relies on the study of \emph{wedges} on $X$, namely, maps
$\Phi\colon \Spec \C[[s,t]] \to X$. 
Wedges can be thought as
arcs on the arc space $J_\infty(X)$, by viewing $t$ as the parameter for the arcs on $X$
and $s$ as a parameter for families of arcs. 
The basic idea is that if $E_i$ and $E_j$ are two distinct prime divisors 
on a resolution $f \colon Y \to X$  such that 
$f_\infty\big(\p_Y^{-1}(E_i)\big)$ is contained in the closure of 
$f_\infty\big(\p_Y^{-1}(E_j)\big)$, then one can detect this by a
wedge $\Phi$ having $\Phi(0,{}_{-}) \in f_\infty\big(\p_Y^{-1}(E_i)\big)$ and 
$\Phi(s,{}_{-}) \in f_\infty\big(\p_Y^{-1}(E_j)\big)$ for $s \ne 0$. 
The Nash problem reduces then to a lifting problem for wedges.  
The existence of such wedges is a deep fact which follows by the Curve Selection Lemma
of Reguera \cite{Re} (see also \cite{FP2}).

In dimension 2, the essential divisorial valuations are those 
determined by the exceptional divisors in the minimal resolution. 
The Nash problem was positively answered in this case by
Fern\'andez de Bobadilla and Pe Pereira \cite{FP}
(we refer to their paper for a list of references on previous results in dimension 2);
the proof of this result 
is a beautiful interplay of topological methods and formal settings. 

\begin{thm}[\cite{FP}]
For any surface, the Nash map is a bijection.
\end{thm}

In higher dimensions, the Nash problem has been positively settled
in a series of cases, see \cite{IK,PPP,LJR}. In general,
however, we have the following negative result.

\begin{thm}[\cite{IK}]
In any dimension $\ge 4$ there are varieties for which the Nash map is not surjective.
\end{thm}

The examples given in \cite{IK}
are based on the following idea. An isolated singularity $O \in X$
is resolved by a sequence of two blow-ups
$Z \to Y \to X$, each extracting a single divisor. 
The divisor $E \subset Z$ extracted by the second blow-up is not birationally ruled, and thus 
must be essential. It is however covered by lines $L$ whose normal bundles $\cN_{L/E}$
have vanishing first cohomology. This implies that  
the image on $Y$ of any arc with contact order 1 with $E$
can be slided away from the image of $E$, 
along the exceptional divisor $F \subset Y$ of the first blow-up.
Therefore the image of $\p_Z^{-1}(E)$ in $J_\infty(X)$ lies in the closure of
the image of $\p_Y^{-1}(F^\o)$ where $F^\o := F \cap Y_\reg$, 
and hence is not dense in any
irreducible component of $\p_X^{-1}(O)$. This means that $\val_E$
is not in the image of the Nash map.

Here we show that the Nash problem has a negative answer in dimension 3 as well.
%Here we extend such examples to dimension 3, thus proving the following result.

\begin{thm}
There are varieties of dimension 3 for which the Nash map is not surjective.
\end{thm}

The reason why in the examples given by Ishii and Koll\'ar
one needs to assume that the dimension 
is at least 4 is that surfaces that are covered by
rational curves are automatically birationally ruled. To give examples in dimension 3, 
we follow a strategy similar to that of Ishii and Koll\'ar, and consider
a 3-dimensional isolated singularity $O \in X$ that
is resolved by a sequence of two blow-ups $Z \to Y \to X$. 
Using this approach, we will construct two examples that, in spite of
being deformations of each other, present different features.

The first example works in the category of algebraic varieties but not in the analytic setting. 
In this example the exceptional divisor $E$
is covered by lines with zero first cohomology of the normal bundle, 
so that its valuation is not in the Nash correspondence; the fact that the valuation 
is essential follows from reasons of 
discrepancies of the two exceptional divisors and 
factoriality (in the Zariski topology) of the first blow-up $Y$. 
The argument on discrepancies 
can in fact be used to give counter-examples in any dimension $\ge 4$
without relying on deep results about certain varieties not being birationally ruled.

The second example is obtained as a degeneration of the first one, 
and works in both categories. 
The degeneration is used to reduce to a case where $Y^\an$ is 
factorial in the analytic topology.
The fact that the valuation associated to $E$ is not in the Nash correspondence
is deduced from the first example, 
by studying the deformation of the associated family of arcs.

\section{First example}
\label{s:eg}

\begin{thm}
\label{t:eg1}
The hypersurface in $\A^4 = \Spec\C[x_1,x_2,x_3,x_4]$ defined by the equation
\[
(x_1^2 + x_2^2 + x_3^2)x_4 + x_1^3 + x_2^3 + x_3^3 + x_4^5 + x_4^6 = 0
\]
gives a counter-example to the Nash problem in the category of algebraic varieties
but not in the category of analytic spaces. 
\end{thm}

\begin{proof}
This hypersurface, which we shall denote by $X$, has an isolated singularity of 
multiplicity 3 at $O = (0,0,0,0) \in \A^4$. 
The blow-up of the maximal ideal
\[
f\colon Y = \Bl_OX \to X
\] 
extracts one exceptional divisor $F$, 
given by the equation $(x_1^2 + x_2^2 + x_3^2)x_4 + x_1^3 + x_2^3 + x_3^3 =0$
in the exceptional divisor $\P^3 = \Proj \C[x_1,x_2,x_3,x_4]$ of $\Bl_O\A^4 \to \A^4$.
One can check that the point $P = (0:0:0:1) \in \P^3$ is the only singular point of $F$. 
In the local chart $U = \Spec\C[u_1,u_2,u_3,x_4] \subset \Bl_0\A^4$
where $u_i = x_i/x_4$, we have $P = (0,0,0,0)$, and $Y$ is defined by 
\[
u_1^2 + u_2^2 + u_3^2 + x_4^2 + u_1^3 + u_2^3 + u_3^3 + x_4^3 = 0.
\]
Then $Y$ has an ordinary double point 
at $P$, and $P$ is the only singular point of $Y$.
The closure of $Y \cap U$ in $\P^4 = \Proj \C[u_0,u_1,u_2,u_3,x_4]$ is a cubic threefold
with only one ordinary double point as singularity, and is thus
factorial in the Zariski topology (see \cite{Cy,CDG}). 
Therefore $Y$ is factorial in the Zariski topology.
The blow-up
\[
g\colon Z = \Bl_PY \to Y
\]
extracts one exceptional divisor $E$ which is a smooth quadric surface, 
defined by $u_1^2 + u_2^2 + u_3^2 + x_4^2= 0$, in the exceptional divisor 
$\P^3 = \Proj \C[u_1,u_2,u_3,x_4]$ of $\Bl_PU \to U$.
We claim that $E$ gives an example of an essential divisorial valuation on $X$ that
is not in the image of the Nash map.

We start by showing that $E$ defines an essential divisorial valuation on $X$
(in the category of algebraic varieties).
This will follow by a discrepancy computation. 
For a prime divisor $D$ on a normal birational model $V \to X$ we denote by
$k_D(X) := \ord_D(K_{V/X})$ the \emph{discrepancy} of $D$ over $X$. 
Here $K_{V/X}$ is the relative canonical divisor, which is 
well-defined because $X$ is normal and Gorenstein (being a hypersurface).
Recall that the discrepancy only depends on the valuation $\val_D$,
so that if $V' \to X$ is a different normal birational model
and the proper transform of $D$ on $V'$ is a divisor $D'$, then 
$k_D(X) = k_{D'}(X)$. In particular, the symbol $k_D(X)$ is defined for a prime divisor
$D$ on any normal model $V$ birational to $X$. 

The discrepancy of a divisor extracted by blowing up a hypersurface singularity
can be computed by blowing up the smooth ambient space and using the adjunction formula. 
In our situation, we have 
\[
k_F(X) = 0 \and k_E(X) = 1.
\]
Bearing in mind that $Z$ is both a resolution of $X$ 
and a resolution of $Y$, we see that
$O \in X$ is a canonical singularity and $Y$ has terminal singularities. 
Denoting by $\fm_O \subset \O_X$ the maximal ideal of $O$, 
one can check from the defining equation that
\[
\val_E(\fm_O) = 1.
\]
Then the fact that $\val_E$ is essential
is a consequence of the following general lemma. 

\begin{lem}
\label{l:essential}
Let $X$ be a 3-dimensional algebraic variety having an isolated canonical 
singularity at a point $O$. 
Suppose that the blow-up of $X$ at $O$ is $\Q$-factorial (in the Zariski topology)
and has terminal singularities, 
and that the (reduced) exceptional divisor $F$ of the blow-up is irreducible.
Then every divisorial valuation $\val_E$ over $X$ 
such that $k_E(X) = 1$ and $\val_E(\fm_O) = 1$ is an essential valuation
(in the category of algebraic varieties).
\end{lem}

\begin{proof}
Suppose by contradiction that $\val_E$ is not essential. Then there is a resolution
of singularities
\[
p \colon X' \to X
\]
such that the center $C \subset X'$ of $\val_E$ is
properly contained in an irreducible component of the exceptional locus $\Ex(p)$.
Since the singularity of $X$ is canonical, every $p$-exceptional divisor 
has nonnegative discrepancy, and therefore 
the relative canonical divisor $K_{X'/X}$ is effective. 
Hence, if $q \colon V \to X$ is another
resolution which factor through $p$, and $r \colon V \to X'$ is the induced map, 
then $K_{V/X} = K_{V/X'} + r^* K_{X'/X} \ge K_{V/X'}$.
This implies that $k_E(X) \ge k_E(X')$. 
By the chain of inequalities
\[
1 = k_E(X) \ge k_E(X') \ge \codim_{X'}(C) - 1 \ge 1
\]
we conclude that $C$ is a curve and $k_E(X) = k_E(X')$.
Then $C$ must be contained in some codimension
one component of $\Ex(p)$, and any such component must have
discrepancy zero over $X$. 
Since the blow-up $Y := \Bl_OX$ is has terminal singularities, 
$\val_F$ is the only divisorial valuation whose discrepancy can be zero.
Therefore $\val_F$ must have discrepancy zero, and 
$C$ is contained in exactly one exceptional
divisor $F'$ of $p$, equal to the proper transform of $F$
(that is, such that $\val_{F'} = \val_F$). 

The assumption that $\val_E(\fm_O) = 1$
implies that $p^{-1}\fm_O \.\O_{X'} = \fa\.\O_{X'}(-F')$
where $\fa \subset \O_{X'}$ is an ideal sheaf whose vanishing locus 
does not contain $C$. After further blowing up $\fa$, we may  
assume without loss of generality that $\fa$ is locally principal: this reduction step is
allowed because it does not
affect the fact that the center of $\val_E$ is a curve contained in the
proper transform of $F$ and in no other divisor that is exceptional over $X$. 
Then, by the universal property of the blow-up, $p$ factors through a morphism 
\[
h\colon X' \to Y.
\]

Note that $h(C) = P$. Since $F'$ is the only $p$-exceptional divisor containing $C$,
and $h(F') = F$, the curve $C$ is not contained in any $h$-exceptional
divisor. Therefore $C$ is an irreducible component of $\Ex(h)$. 
By Lemma~\ref{l:small}, this contradicts the fact that $Y$ is $\Q$-factorial. 
We conclude that $\val_E$ is an essential valuation on $X$.
\end{proof}

The reason why the valuation is not essential in the analytic category
is that the analytic space $Y^\an$ associated to $Y$ 
admits small resolutions in the analytic category (cf.\ Example~\ref{eg:small}). 
Such models 
are proper over $X^\an$ and thus give analytic resolutions of singularities of $X^\an$.
If $Y' \to Y^\an$ is one of these small resolutions, then the exceptional 
curve $C' \subset Y'$ is equal to the center of $\val_E$, and is fully contained in the
proper transform of $F^\an$, which is an exceptional component over $X^\an$.

It remains to check that $\val_E$ is not in the Nash correspondence.
This is immediate from the fact that $\val_E$ is not analytically essential over $X^\an$,
since such property implies that the valuation cannot be in the image of the Nash map
(in either category). 
\end{proof}

\begin{rmk}
\label{r:stronger}
One can also see that $\val_E$ is not in the Nash correspondence directly, without 
passing to the analytic setting, by following the same arguments of \cite{IK}.
We briefly recall the argument. 

Let $\ff \colon \Spec\C[[t]] \to Z$
be any formal arc on $Z$ with contact order 1 along $E$, 
and let 
\[
\f \colon \Spec\C[[t]] \to Y
\] 
be its image on $Y$. 
Let $L \subset E \cong \P^1 \times \P^1$ 
be the line through $\ff(0)$ in one of the two rulings. 
Since $\cN_{L/E} \cong \O_{\P^1}$, we have 
$H^1(L,\cN_{L/E}) = 0$. Thus \cite[Lemma~4.2]{IK} 
(cf.~Lemma~\ref{l:IK} below) applies, and the arc $\f$ extends to 
a smooth wedge 
\[
\Phi \colon \Spec\C[[s,t]] \to Y.
\]
The arc $\f$ has base at $P$ but is not fully contained in $F$.
Since $F$ is a Cartier divisor, the pull-back $\Phi^*F$ is
a curve on $\Spec\C[[s,t]]$. By translating the arc 
$\{s=0\} \subset \Spec\C[[s,t]]$ to the generic point of this curve, 
we obtain an arc on $Y$ (over a one-dimensional function field)
with finite order contact along $F$ away from $P$, which specializes to $\f$. 
Letting $F^\o := F \cap Y_\reg = F \smallsetminus \{P\}$, this implies that 
$g_\infty\big(\p_Z^{-1}(E)\big)$ 
is contained in the closure of $\p_Y^{-1}(F^\o)$ in $J_\infty(Y)$,
and therefore $(f\o g)_\infty\big(\p_Z^{-1}(E)\big)$ is contained
in the closure of $f_\infty\big(\p_Y^{-1}(F^\o)\big)$ in $J_\infty(X)$.
This proves that $\val_E$ does not correspond to any irreducible component of 
the space of arcs in $X$ through $O$. 

It should be noted that the property that $g_\infty\big(\p_Z^{-1}(E)\big)$ is contained
in the closure of $\p_Y^{-1}(F^\o)$ is {\it a priori} a stronger property than 
having an inclusion of $(f\o g)_\infty\big(\p_Z^{-1}(E)\big)$ 
in the closure of $f_\infty\big(\p_Y^{-1}(F^\o)\big)$, and it cannot be deduced, for instance, 
by simply knowing that $\val_E$ is not essential, in the analytic category, over $X^\an$.
This stronger property will come useful in the study of the second example. 
\end{rmk}

\section{Second example}
\label{s:eg2}

\begin{thm}
\label{t:eg2}
The hypersurface in $\A^4 = \Spec\C[x_1,x_2,x_3,x_4]$ 
defined by the equation
\[
(x_2^2 + x_3^2)x_4 + x_1^3 + x_2^3 + x_3^3 + x_4^5 + x_4^6 = 0
\]
gives a counter-example to the Nash problem in both categories
of algebraic varieties and analytic spaces. 
\end{thm}

\begin{proof}
Consider the family of
affine hypersurfaces in $\A^4 = \Spec\C[x_1,x_2,x_3,x_4]$ defined by
\[
(\lambda x_1^2 + x_2^2 + x_3^2)x_4 + x_1^3 + x_2^3 + x_3^3 + x_4^5 + x_4^6 = 0, 
\quad \lambda \in \C.
\]
Let $\X$ denote the total space of the family, which we view as a scheme over $\A^1$.
The hypersurface studied in the previous section appears in this family when $\lambda =1$; 
the same conclusions we draw for $\X_1$ 
hold however for $\X_\lambda$ for every $\lambda \ne 0$.
Here we focus on the central fiber $\X_0$ of the deformation.

Consider the blow-up 
\[
f\colon \Y = \Bl_\cO\X \to \X,
\] 
where $\cO = \{O\} \times \A^1$, and let $\cF$ be its exceptional divisor.
Consider then the section $\cP = \{P\} \times \A^1 \subset \Y$, where $P = (0,0,0,0)$ 
in the local chart $U = \Spec\C[u_1,u_2,u_3,x_4] \subset \Bl_O\A^4$
with $u_i = x_i/x_4$, and let
\[
g\colon \Z = \Bl_\cP\Y \to \Y
\]
be its blow-up, with exceptional divisor $\cE$.
Since the multiplicity of $\X_\lambda$ at $\O_\lambda$ is the same for all $\lambda$, 
each fiber $\Y_\lambda$ is the blow-up of $\X_\lambda$ at $\O_\lambda$. 
Similarly, the multiplicity of $\Y_\lambda$ at $\cP_\lambda$ is independent of $\lambda$, 
and each $\Z_\lambda$ is the blow-up of $\Y_\lambda$ at $\cP_\lambda$. 

For simplicity, throughout this section we shall denote 
\[
X := \X_0, \quad Y := \Y_0, \quad Z := \Z_0, \quad F := \cF_0, \quad E := \cE_0.
\]
Then $Y$ is defined by 
\[
u_2^2 + u_3^2 + x_4^2 + u_1^3 + u_2^3 + u_3^3 + x_4^3 = 0
\]
in the chart $U$, and $P$ is the only singular point
of $Y$. The exceptional divisor $E$ extracted by the blow-up
$Z = \Bl_PY \to Y$ is the quadric cone of equation $u_2^2 + u_3^2 + x_4^2= 0$ 
and vertex at $Q = (1:0:0:0)$ in $\P^3 = \Proj \C[u_1,u_2,u_3,x_4]$.
In the local chart $V = \Spec\C[u_1,v_2,v_3,v_4] \subset \Bl_P U$
where $v_i = u_i/u_1$ (here we set $u_4=x_4$), we have $Q = (0,0,0,0)$, and $Z$ is defined by 
\[
u_1 + v_2^2 + v_3^2 + v_4^2 + u_1(v_2^3 + v_3^3 + v_4^3) = 0.
\]
This shows that $Z$ is smooth at $Q$. 
Since there is no singular point in any of the other charts, $Z$ is a resolution of $X$.
Note, in particular, that every fiber $Z_\lambda$, for $\lambda \in \A^1$, is regular, 
and $\Z$ is flat over $\A^1$. Therefore $\Z$ is smooth over $\A^1$.

\begin{lem}
$\val_E$ is essential both in the algebraic category and in the analytic category.
\end{lem}

\begin{proof}
The fact that $\val_E$ is essential in 
the algebraic category follows by the same arguments of the previous section. 
The proof that the valuation is
also essential in the analytic category is similar, and only 
requires some adaptation.

First, we claim that $Y^\an$ is locally analytically factorial at $P$.
To see this, let $Y^\o \subset Y^\an$ 
be an arbitrary Euclidean open neighborhood of $P$, 
and let $Z^\o \to Y^\o$ denote the restriction of $Z^\an \to Y^\an$.
Let $D$ be any divisor on $Y^\o$, and let $D'$ be its proper transform in $Z^\o$.
Recall that $E^\an$ is isomorphic to a singular quadric hypersurface in $\P^3$
and its Picard group is generated by the hyperplane class $\O_{E^\an}(1)$. 
The normal bundle of $E^\an$ in $Z^\o$ is isomorphic to $\O_{E^\an}(-1)$ 
(this can be checked by computing the normal
bundle of the exceptional divisor extracted by the blow-up of $\Bl_O\A^4$ at $P$
and restricting to $Z$). 
Then $\O_{Z^\o}(D' + mE^\an)$ restricts to the trivial line bundle on $E^\an$
for some positive integer $m$. 
Since $Z^\o \to Y^\o$ is a rational resolution, it follows by the same
arguments in the proof of \cite[Proposition~(12.1.4)]{KM92} 
that $\O_{Z^\o}(D' + mE^\an)$ restricts to the trivial bundle
on the inverse image of a small contractible 
neighborhood of $P \in Y^\o$. This implies that $\O_{Z^\o}(D' + mE^\an)$ 
is the pull-back of a line bundle on $Y^\o$, and therefore
that $D$ is Cartier. 

Suppose that $\val_E$ is not essential over $X^\an$. 
Let $p \colon X' \to X^\an$ be a resolution of singularities
such that the center $C$ of $\val_E$ is
strictly contained in some irreducible component of the exceptional locus $\Ex(p)$. 

The computations of discrepancies are the same as in the algebraic setting, and
the same argument as in the proof of Lemma~\ref{l:essential}
shows that $C$ is a curve, the proper transform $F'$ of $F^\an$ is the
only component of $\Ex(p)$ containing $C$, and writing
$p^{-1}\fm_O \.\O_{X'} = \fa \. \O_{X'}(-F')$, 
$C$ is not contained in the vanishing locus of $\fa$
(since $C$ is actually projective, $\fa$ can only vanish at
finitely many points on $C$). 
The ideal sheaf $p^{-1}\fm_O \.\O_{X'}$ (viewed here as a sheaf of ideals in the 
analytic topology) is finitely generated by global holomorphic functions on $X'$, 
and so $\fa$ is finitely generated by global sections of 
$\O_{X'}(F')$. We can thus blow up $\fa$ and further resolve the
singularities. After performing this reduction, we can assume without loss of
generality that $p^{-1}\fm_O \.\O_{X'}$ is an invertible sheaf. 
The universal property of the blow-up
can be applied to our setting and therefore $p$ factors through
an analytic map $h \colon X' \to Y^\an$. 
Just like in the algebraic case, $C$ is a one-dimensional irreducible component of 
the exceptional locus of $h$.
By Lemma~\ref{l:small}, this implies that $Y^\an$ is not $\Q$-factorial in 
the analytic topology. 

We get a contradiction, and therefore $\val_E$ must be essential 
in the analytic setting, too.
\end{proof}

To prove that $\val_E$ is not in the image of the Nash map,
we reduce to use what we already know about the first 
example, by looking at the family of 
arcs associated to $\val_E$ as limits of families of arcs on the nearby fibers $\cX_\lambda$. 
In order to formalize this, we work in the category of schemes. 
As explained in Section~\ref{s:analytic},  
the property of a divisorial valuation being in the 
image of the Nash map is independent of the category we work in, 
so it suffices to study the problem in the category of schemes. 

In the following, let for short 
\[
\F^\o := \F \smallsetminus \{\cP\} \and F^\o := F \smallsetminus \{P\}.
\]
Consider the images of the sets $\p_Z^{-1}(E)$ and $\p_Y^{-1}(F^\o)$ in $J_\infty(X)$. 
We need to check that the first image is contained in the closure of the latter.

The idea is simple: if, for $\lambda \ne 0$, one family of arcs is a 
specialization of the other family, then the same should hold, by degeneration, 
on the central fiber $X = \X_0$.
Specifically, we know by the previous example that, for $\lambda \ne 0$, 
$(f_\lambda\o g_\lambda)_\infty\big(\p_{\cZ_\lambda}^{-1}(\cE_\lambda)\big)$ 
is contained in the closure of the image of 
$(f_\lambda)_\infty\big(\p_{\cY_\lambda}^{-1}(\cF_\lambda^\o)\big)$. 
We would like to show that this property is preserved when 
we let $\lambda$ degenerate to $0$. 
For this to work, we need to know
that the closure of $(f_\lambda)_\infty\big(\p_{\cY_\lambda}^{-1}(\cF_\lambda^\o)\big)$
degenerates to the closure of $(f_0)_\infty\big(\p_Y^{-1}(F^\o)\big)$.
This is the key point that we need to prove. 

At least intuitively, there is a good reason why it should be expected. The point is that, 
for every $\lambda$, we have 
\[
\p_{\X_\lambda}^{-1}(\cO_\lambda) = 
\Big((f_\lambda\o g_\lambda)_\infty\big(\p_{\cZ_\lambda}^{-1}(\cE_\lambda)\big)\Big) \cup
\Big((f_\lambda)_\infty\big(\p_{\cY_\lambda}^{-1}(\cF_\lambda^\o)\big)\Big),
\]
and even though the two sets in the right hand side (or, more precisely, their closures)
have infinite dimension, their codimensions in $J_\infty(\X_\lambda)$
are finite. 
(These codimensions can be computed in terms of the Mather discrepancies of the two divisors,
see \cite[Theorem~3.9]{dFEI}.) 
For every $\lambda$, 
the codimension of $(f_\lambda\o g_\lambda)_\infty\big(\p_{\cZ_\lambda}^{-1}(\cE_\lambda)\big)$
is strictly larger than the codimension of
$(f_\lambda)_\infty\big(\p_{\cY_\lambda}^{-1}(\cF_\lambda^\o)\big)$. 
If we knew that when we let $(f_\lambda)_\infty\big(\p_{\cY_\lambda}^{-1}(\cF_\lambda^\o)\big)$
degenerate, the codimension of each irreducible component over $\lambda = 0$
cannot exceed the codimension of a general fiber of the deformation, 
then we could easily conclude that the central fiber is irreducible
and must coincide with the closure of $(f_0)_\infty\big(\p_Y^{-1}(F^\o)\big)$, 
and that $(f_0\o g_0)_\infty\big(\p_Z^{-1}(E)\big)$ must be contained in it. 
It is however unclear to us
whether such semi-continuity property holds in the infinite dimensional setting. 

For this reason, we reduce to work 
in the finite dimensional setting, by taking images at the finite jet levels.
A technical difficulty arises:
due to the singularity of $X$, the fibers $\p_{X,m}^{-1}(O)$
(where $\t_{X,m} \colon J_m(X) \to X$ denote the canonical projection)
have larger than expected dimensions, and the maps
$\p_X^{-1}(O) \to \t_{X,m}^{-1}(O)$ are not dominant. 
So, we work over $\Y$, which is less singular, and use the 
inclusion of $(g_\lambda)_\infty\big(\p_{\cZ_\lambda}^{-1}(\cE_\lambda)\big)$ 
in the closure of $\p_{\cY_\lambda}^{-1}(\cF_\lambda^\o)$
(which holds for $\lambda \ne 0$, see Remark~\ref{r:stronger}). 
This will give us enough room to compensate for the extra dimensions
in our computations.

Closing this digression, 
let us consider the relative arc space $J_\infty(\X/\A^1)$
and the relative jet spaces $J_m(\X/\A^1)$ of $\X$ over $\A^1$, which
respectively parameterize commutative diagrams
\[
\xymatrix{
\Spec\C[[t]] \ar[d]\ar[r]^(.63)\f &\X \ar[d] 
&& \Spec\C[t]/(t^{m+1}) \ar[d]\ar[r]^(.7)\g &\X \ar[d]\\
\Spec \C \ar[r] &\A^1 && \Spec \C \ar[r] &\A^1
}.
\] 
We denote the canonical projections by
\[
\p_{\X/\A^1,m} \colon J_\infty(\X/\A^1) \to J_m(\X/\A^1)
\and \p_{\X/\A^1} \colon J_\infty(\X/\A^1) \to \X.
\]
We define similar spaces and maps for $\Y$ and $\Z$.
Note that for every $\lambda \in \A^1$ there are natural 
identifications $J_\infty(\X/\A^1)_\lambda = J_\infty(\X_\lambda)$
and $J_m(\X/\A^1)_\lambda = J_m(\X_\lambda)$, and similarly for $\Y$ and $\Z$.

There is a commutative diagram
\[
\xymatrix{
J_\infty(\Z/\A^1) \ar[d]^{\p_{\Z/\A^1}} \ar[r]^{g_\infty} 
& J_\infty(\Y/\A^1) \ar[d]^{\p_{\Y/\A^1}} \ar[r]^{f_\infty} 
& J_\infty(\X/\A^1) \ar[d]^{\p_{\X/\A^1}} \\
\Z \ar[r]^g & \Y \ar[r]^g & \X
}
\]
where $g_\infty$ and $f_\infty$ are canonically induced by $g$ and $f$, and restrict
to each fiber over $\A^1$ to the corresponding maps $(g_\lambda)_\infty$
and $(f_\lambda)_\infty$.
We have
\[
\big(\p_{\Z/\A^1}^{-1}(\E)\big)_\lambda = \p_{\Z_\lambda}^{-1}(\E_\lambda)
\and
\big(\p_{\Y/\A^1}^{-1}(\F^\o)\big)_\lambda = \p_{\Y_\lambda}^{-1}(\F^\o_\lambda)
\fall \lambda.
\]
If $\lambda \ne 0$, then $(g_\lambda)_\infty\big(\p_{\Z_\lambda}^{-1}(\E_\lambda)\big)$
is contained in the closure of $\p_{\Y_\lambda}^{-1}(\F^\o_\lambda)$
(see Remark~\ref{r:stronger}), 
and since $\Z$ is smooth over $\A^1$, $\p_{\Z/\A^1}^{-1}(\E)$ is irreducible. 
Therefore $g_\infty\big(\p_{\Z/\A^1}^{-1}(\E)\big)$ is contained in
the closure of $\p_{\Y/\A^1}^{-1}(\F^\o)$.

\begin{lem}
\label{l:closure-restr}
The fiber over $0 \in \A^1$ of the closure of $\p_{\Y/\A^1}^{-1}(\F^\o)$
in $J_\infty(\Y/\A^1)$ is equal to the closure of $\p_Y^{-1}(F^\o)$
in $J_\infty(Y)$. That is:
\[
\Big(\ov{\p_{\Y/\A^1}^{-1}(\F^\o)}\Big)_0 = \ov{\p_Y^{-1}(F^\o)}.
\]
\end{lem}

\begin{proof}
For short, let
\[
\cS_m := \ov{\p_{\Y/\A^1,m}\big(\p_{\Y/\A^1}^{-1}(\F^\o)\big)}
\and
S_m := \ov{\p_{Y,m}\big(\p_Y^{-1}(F^\o)\big)}
\]
where the closures 
are taken in the respective jet spaces $J_m(\Y/\A^1)$ and $J_m(Y)$. 
By the definition of the inverse limit topology on arc spaces, we have
\[
\ov{\p_{\Y/\A^1}^{-1}(\F^\o)} 
= \bigcap_m \p_{\Y/\A^1,m}^{-1}(\cS_m)
\and
\ov{\p_Y^{-1}(F^\o)} = \bigcap_m \p_{Y,m}^{-1}(S_m),
\]
It is therefore enough to show that for every $m$ the fiber 
of $\cS_m$ over $0 \in \A^1$
is equal to $S_m$. The inclusion $S_m \subset (\cS_m)_0$ is clear, and we 
need to show the reverse inclusion.

Suppose by contradiction that $S_m \subsetneq (\cS_m)_0$, and
let $T$ be an irreducible component of $(\cS_m)_0$ that is not equal to $S_m$. 
Since $(\cS_m)_0$ agrees with $S_m$ over $Y_\reg$, 
$T$ must be contained in the fiber over $P \in Y$. 

For $\lambda \ne 0$, the set $(\cS_m)_\lambda$ is equal to the closure 
of $\t_{\Y_\lambda,m}^{-1}(\cF^\o_\lambda)$ where
$\t_{\Y_\lambda,m} \colon J_m(\Y_\lambda) \to \Y_\lambda$ is the canonical projection. 
Note that $\t_{\Y_\lambda,m}^{-1}(\cF^\o_\lambda)$ is an 
irreducible codimension 1 subset of $J_m((\Y_\lambda)_\reg)$, which is
$3(m+1)$-dimensional, hence it has dimension $3(m+1)-1$.
Therefore $(\cS_m)_\lambda$, for $\lambda \ne 0$, 
is irreducible and has dimension $3(m+1)-1$. 
Since $\cS_m \to \A^1$ is a surjective morphism from a variety, 
the dimension of $T$ is at least the dimension of a general fiber
over $\A^1$, and hence
\[
\dim T \ge 3(m+1) - 1.
\]

Consider the set
\[
\p_{\~\A^4,m}^{-1}(T) \subset J_\infty(\~\A^4)
\]
where we denote for short $\~\A^4 := \Bl_O\A^4$.
This set is contained in the fiber over $P \in \~\A^4$ and has
codimension $\le m+2$ in $J_\infty(\~\A^4)$.
Denote by $\cI_Y$ the ideal sheaf of $Y$
and by $\fm_P$ the maximal ideal of $P$ in $\~\A^4$, fix $n \ge m$, 
and let $P_Y^{(n)} := V(\cI_Y + \fm_P^{n+1}) \subset \~\A^4$ 
be the $n$-th neighborhod of $P$ in $Y$.
Since $J_m(P_Y^{(n)})$ and $J_m(Y)$ have the same fiber over $P$, we have
$T \subset J_m(P_Y^{(n)})$.
Then $\p_{\~\A^4,m}^{-1}(T)$ is contained in the \emph{contact locus}
\[
\p_{\~\A^4,m}^{-1}\big(J_m(P_Y^{(n)})\big) = 
\{\a \in J_\infty(\~\A^4) \mid \val_\a(\cI_Y + \fm_P^{n+1}) \ge m + 1\}.
\]
Let $C$ be an irreducible component of this contact locus 
that contains $\p_{\~\A^4,m}^{-1}(T)$. 
Note that $C$ lies over $P$ and has codimension $\le m+2$.
It follows from \cite[Theorem~A]{ELM04} that there is a prime exceptional divisor
$D$ over $\~\A^4$ with center $P$, and a positive integer $q$, 
such that 
\[
q\.\big(k_D(\~\A^4) + 1\big) = \codim\big(C,J_\infty(\~\A^4)\big) \le m+2
\]
and
\[
q \. \val_D(\cI_Y + \fm_P^{n+1}) = \val_C(\cI_Y + \fm_P^{n+1}) \ge m+1.
\]
Therefore the divisor $D$ has \emph{log discrepancy} 
\[
a_D(\~\A^4,P_Y^{(n)}) := k_D(\~\A^4) + 1 - \val_D(\cI_Y + \fm_P^{n+1}) \le 1
\]
over the pair $(\~\A^4,P_Y^{(n)})$.

This is however impossible because $P_Y^{(n)} \subset Y$ and
the pair $(\~\A^4,Y)$ has \emph{minimal log discrepancy} 2 at $P$, which can be checked 
as follows (we refer to \cite{Amb} for the definition and general properties of  
minimal log discrepancies).
Consider the sequence of two blow-ups $(\~\A^4)'' \to (\~\A^4)' \to \~\A^4$, where
the first blow-up is centered at the point $P$ (hence the proper transform
of $Y$ is equal to $Z$), and the second blow-up is centered at the point $Q$. 
The pull-back of $Y$ on $(\~\A^4)''$ is a simple normal crossing divisor,
and the log discrepancies of the two exceptional divisors over $(\~\A^4,Y)$ 
are 2 and 4, respectively. 
The minimal log discrepancy is just the minimum of these two numbers.

This completes the proof of the lemma.
\end{proof}

We can now finish the proof of the theorem. 
We saw that the image of $\p_{\Z/\A^1}^{-1}(\E)$ in $J_\infty(\Y/\A^1)$ being contained in
the closure of $\p_{\Y/\A^1}^{-1}(\F^\o)$.
By the lemma, if we restrict this inclusion 
to the fiber over $0 \in \A^1$ we then obtain that
the image of $\p_Z^{-1}(E)$ in $J_\infty(Y)$ is contained in the closure of $\p_Y^{-1}(F^\o)$.
Mapping down to $J_\infty(X)$, we conclude that 
the image of $\p_Z^{-1}(E)$ in $J_\infty(X)$ is contained in the closure of image of
$\p_Y^{-1}(F^\o)$. 
This completes the proof that $\val_E$ is not in the Nash correspondence.
\end{proof}

\section{On Ishii--Koll\'ar's smooth wedge construction}
\label{s:IK}

This section is devoted to a discussion of the following lemma due to Ishii and Koll\'ar.
The discussion given below 
provides a different viewpoint on this interesting property
which might lead to find more general formulations. 

While the lemma can be avoided in the proof of Theorem~\ref{t:eg1},  
it leads to a stronger property that is used in the proof of Theorem~\ref{t:eg2}. 
More precisely, in the notation of the proof of Theorem~\ref{t:eg1}, 
the lemma implies that 
$g_\infty\big(\p_Z^{-1}(E)\big)$ is contained
in the closure of $\p_Y^{-1}(F^\o)$. 
This property does not follows, formally, from
the inclusion of $(f\o g)_\infty\big(\p_Z^{-1}(E)\big)$ 
in the closure of $f_\infty\big(\p_Y^{-1}(F^\o)\big)$.

\begin{lem}[\protect{\cite[Lemma~4.2]{IK}}]
\label{l:IK}
Let $Y \subset \A^{n+1}$ be a hypersurface with an isolated singularity
at a point $P$, and suppose that the exceptional divisor $E$ of the blow-up 
\[
Z = \Bl_PY \to Y
\]
is a reduced, irreducible hypersurface in the exceptional 
divisor $\P^n$ of $\Bl_P\A^{n+1} \to \A^{n+1}$. 
Let $\ff \colon \Spec\C[[t]] \to Z$ be an arc with contact order 1 along $E$, 
and assume that there is a line $L \subset E \subset \P^n$
through $\ff(0)$ such that $H^1(L,\cN_{L/E}) = 0$. Then 
the image $\f \colon \Spec\C[[t]] \to Y$ of $\ff$ in $Y$ extends to 
a smooth wedge $\Phi \colon \Spec\C[[s,t]] \to Y$. 
\end{lem}

The proof given in \cite{IK} goes by constructing the wedge directly on $Y$, 
by lifting solutions modulo powers of $(s,t)$. 
Here we discuss an alternative approach by means of formal geometry:
the rough idea is to construct a blown-up wedge on $Z$
in such a way that the map induced by $Z \to Y$ is 
just the contraction of a $(-1)$-curve to a smooth wedge.
As explained below, our approach does not allow us to prove the
full strength of the lemma. 

Since $L \cong \P^1$, by Birkhoff--Grothendieck theorem $\cN_{L/E}$ decomposes 
as the direct sum of line bundles. The vanishing of the first cohomology of $\cN_{L/E}$ 
implies that each of these lines bundles has degree $\ge -1$, and
the fact that $\cN_{L/E}$ injects into $\cN_{L/\P^n} \cong \O_L(1)^{\oplus n-1}$ 
implies that all degrees are $\le 1$. Therefore we can write the decomposition as follows:
\[
\cN_{L/E} \cong \bigoplus_{i=2}^{n-1} \O_L(-a_i)  \where -1 \le a_i \le 1.
\]

In the following, we shall assume that $a_i \ge 0$ for all $i$. This condition is satisfied
in the first example discussed in this paper; it is also
satisfied whenever $E$ is a general hypersurface of degree $n-1$ in $\P^n$
and $L$ is a general line in $E$, a case which suffices to construct counter-examples
to the Nash problem in all dimensions $\ge 3$. 

We have $\cN_{E/Z}|_L = \O_L(-1)$ since $L$ is a line
in $\P^n$, and thus $\Ext^1(\cN_{E/Z}|_L,\cN_{L/E}) = 0$, which yields the splitting 
of the normal bundle
\[
\cN := \cN_{L/Z} \cong \O_L(-1) \oplus \O_L(-a_2) \oplus \dots \oplus \O_L(-a_{n-1}).
\]
Let $N := \Spec S(\cN^*)$ denote the total space of the normal bundle $\cN$, and
let $L_N^{(\infty)}$ and $L_Z^{(\infty)}$ be the formal neighborhoods of $L$ in $N$ and $Z$. 
The obstructions to construct an isomorphism 
between the two formal neighborhoods are in the groups
$H^1(L,\cN \otimes S^d(\cN^*))$ for $d \ge 2$,
see \cite[Remark~4.6]{ABT}, which in our case are all trivial.
We can therefore choose an isomorphism of formal neighborhoods
\[
\t \colon L_N^{(\infty)} \xrightarrow{\,\cong\,} L_Z^{(\infty)}.
\]
The scheme $L_N^{(\infty)}$ is covered by two affine charts
\[
L_N^{(\infty)} = \Spf\C[z][[v_1,v_2,\dots,v_{n-1}]] \cup 
\Spf\C[1/z][[v_1z,v_2z^{a_2},\dots,v_{n-1}z^{a_{n-1}}]].
\]
Here the variable $z$ is an affine parameter along $L$, and
$v_1$ (in the first chart) corresponds 
to the frame induced by the first summand $\O_L(-1)$. 
Suppose that $\ff(0)$ has coordinate $z=0$ in $L$. 
The restriction of the arc $\f$ to 
the formal neighborhood of $L$ can be written as an $n$-ple of power series
\[
(z(t),v_1(t),\dots,v_{n-1}(t)) \in (\C[[t]])^n
\]
where $z(0) = v_1(0) = \dots = v_{n-1}(0) = 0$, and $v_1'(0) \ne 0$.
The nonvanishing of $v_1'(0)$ is a consequence of the fact that
the arc $\ff$ has order of contact 1 with $E$.

Consider then the formal neighborhood $L_M^{(\infty)}$ of $L$ in the 
total space $M$ of the line bundle $\O_L(-1)$. We can write
\[
L_M^{(\infty)} = \Spf \C[s/t][[t]] \cup \Spf \C[t/s][[s]].
\]
Setting $z = s/t$ and $v_i = v_i(t)$, we obtain compatible maps
\[
\C[z][[v_1,v_2,\dots,v_{n-1}]] \to \C[s/t][[t]], \quad
\C[1/z][[v_1z,v_2z^{a_1},\dots,v_{n-1}z^{a_{n-1}}]] \to \C[t/s][[s]].
\]
These are well defined because $t$ divides $v_i(t)$ for all $i$.
By gluing together and composing with $\t$, we obtain a morphism 
\[
\Psi \colon L_M^{(\infty)} \to Z
\]
whose image contains every truncated jet of $\ff$.
The fact that $v_1'(0) \ne 0$ implies that $\Psi$ is injective and that pulls 
$L$ back to the zero section of $L_M^{(\infty)}$. Note on the other hand that we have
a morphism
\[
\s \colon L_M^{(\infty)} \to \Bl_{(0,0)}\Spec\C[[s,t]] \to \Spec\C[[s,t]]
\]
given in the two charts of $L_M^{(\infty)}$
by the inclusions $\C[[s,t]][s/t] \subset \C[s/t][[t]]$
and $\C[[s,t]][t/s] \subset \C[t/s][[s]]$. Since
$\s(L_M^{(\infty)})$ contains all finite neighborhoods $\Spec\C[s,t]/(s,t)^k$
of the origin in $\Spec\C[[s,t]]$, $\Psi$ induces a map 
\[
\Phi \colon \Spec\C[[s,t]] \to Y 
\]
which, by construction, is a smoothly embedded wedge extending the arc $\phi$.

\section{Koll\'ar's examples}

After the first version of this paper was made public, 
more counter-examples in dimension 3 were found by Koll\'ar \cite{Kol2}. 
This section was added, following a suggestion by one of the referees, 
to compare our examples and technique of proof to those of Koll\'ar.

The paper \cite{Kol2} studies 3-dimensional $cA_1$-type singularities, which are
locally defined by
\[
x_1^2 + x_2^2 + x_3^2 + x_4^m = 0, \quad m \ge 2.
\]
Let $X_m$ denote such hypersurface, and let $O \in X_m$ be the singular point.
It was shown in \cite{Na} that these singularities have at most two essential
valuations, and the precise count 
(exactly two essential valuations if $m$ is odd $\ge 5$, only one
otherwise) is obtained by 
a generalization of Lemma~\ref{l:essential}. 
The Morse lemma with parameters is then used, 
in combination with an inductive argument on the number of blow-ups resolving 
the singularity, to prove that for, every $m \ge 2$, the set $\p_{X_m}^{-1}(O)$ 
is irreducible. It follows that $X_m$ gives a counter-example
to the Nash problem for all odd $m \ge 5$. 

Comparing our methods to those of \cite{Kol2}, the main difference
lies in the proof of the irreducibility of the family of arcs through the singularity.

\end{document}